\documentclass[12pt]{amsart}

\usepackage{fw}

\begin{document}

\begin{abstract}
    Let $\bfv_1,\ldots,\bfv_m$ be points in a metric space with distance $d$, and let $w_1,\ldots,w_m$ be positive real weights. The weighted Fermat-Weber points are those points $\bfx$ which minimize $\sum w_i d(\bfv_i, \bfx)$. 
    We extend a result of Com\u{a}neci and Joswig, that the set of unweighted Fermat-Weber points agrees with the ``central" covector cell of the tropical convex hull of $\bfv_1,\ldots,\bfv_m$, to the weighted setting.
    In particular, we show that for any fixed data points $\bfv_1, \ldots, \bfv_m$, and any covector cell of the tropical convex hull of the data, there is a choice of weights that makes that cell the Fermat-Weber set.
    We similarly extend the method of Com\u{a}neci and Joswig for computing consensus trees in phylogenetics.
\end{abstract}

\maketitle

\section{Introduction}

Given data points $\bfv_1, \ldots, \bfv_m$ in a metric space $X$ with distance $d$, a point $\bfx \in X$ that minimizes the average distance to the data is called a \textit{Fermat-Weber point} of $\bfv_1, \ldots, \bfv_m$. 
In general, Fermat-Weber points are not unique.
Thus, we study the set of all Fermat-Weber points of $\bfv_1, \ldots, \bfv_m$, which is called the \textit{Fermat-Weber set} of $\bfv_1, \ldots, \bfv_m$.
In \cite{joswig-com-tropical-medians}, Comǎneci and Joswig study the Fermat-Weber set for data in the \textit{tropical projective torus}, $\TP$, with the \textit{asymmetric tropical metric} $d_\Delta$.
They prove that Fermat-Weber set with respect to $d_\Delta$ is contained in the tropical convex hull of the data, $\mathrm{tconv}\{ \bfv_1,\ldots,\bfv_m \}$.
In contrast, the Fermat-Weber set with respect to the usual (symmetric) tropical metric is typically not contained in $\mathrm{tconv}\{ \bfv_1,\ldots,\bfv_m \}$ \cite{lin2017convexity,lin-yoshida-tropical-fw}.

In this paper, we study the \textit{weighted} Fermat-Weber set for data in $\TP$ with the asymmetric tropical metric $d_\Delta$.
In the weighted variant of the Fermat-Weber problem, the input also includes a positive real weight, $w_i$, for each data point $\bfv_i$. 
Then, the \textit{weighted Fermat-Weber set} is the collection of all points $\bfx \in X$ which minimize the sum $\sum_{i=1}^m w_i d(\bfx, \bfv_i)$.
Our main theorem, stated below, generalizes \cite[Theorem 4]{joswig-com-tropical-medians} to the weighted setting.

\begin{theorem}
    \label{thm:main}
    Given data $\bfv_1, \ldots, \bfv_m \in \TP$ and any real positive weights $w_1, \ldots, w_m$, the weighted Fermat-Weber set with respect to $d_\Delta$ is contained in $\tconv\{ \bfv_1,\ldots,\bfv_m \}$. More precisely, it is a covector cell of $\tconv\{ \bfv_1,\ldots,\bfv_m \}$. Furthermore, for every covector cell of $\tconv\{ \bfv_1,\ldots,\bfv_m \}$ there exist positive real weights $w_1, \ldots, w_m$ so that the weighted Fermat-Weber set is that cell.
\end{theorem}

In addition to the tropical convexity properties of these Fermat-Weber sets, applications to phylogenetics further motivate the study of the weighted Fermat-Weber problem in the asymmetric tropical setting.
Phylogenetics is a subfield of computational biology concerned with determining evolutionary relationships among present-day species.
One central problem in phylogenetics, called \textit{phylogenetic (tree) reconstruction}, is the following:
given biological data, e.g.\, DNA alignments, from $n$ taxa, output a tree on $n$ leaves that represents the evolutionary history of the $n$ taxa.
There are many methods for phylogenetic reconstruction \cite{de2014phylogenetic}, and one fundamental problem is that different methods may return different trees, even on the same input data.
A \textit{consensus method} is an algorithm that attempts to determine which features in the various reconstructed trees are likely to be present in the true tree.
The output is a \textit{consensus tree} with those features.
For a survey of consensus methods, see \cite{bryant}.

The \textit{space of equidistant phylogenetic trees} on $N$ leaves is a tropically convex space under a certain embedding \cite{ardila2006bergman, ITG}.
Thus, \cite[Theorem 4]{joswig-com-tropical-medians} implies that the (weighted) Fermat-Weber points of tree data with respect to $d_\Delta$ also represent trees; our \Cref{thm:main} extends this property to the weighted case.
In \cite{joswig-com-tropical-medians}, Comǎneci and Joswig then define a consensus method using Fermat-Weber points with respect to $d_\Delta$.
They prove this consensus method is well-behaved in the sense that (1) it is regular (as defined in \cite{bryant2017can}), and (2) it is Pareto and co-Pareto on rooted triples.
We analogously define a consensus method using weighted Fermat-Weber points (see \Cref{defn:consensus}), which is regular, and both Pareto and co-Pareto on rooted triples.

Moreover, the extension to the \textit{weighted} setting is important in certain phylogenetic settings where weighted tree data appears.
For example, in bootstrapping weights are on each tree are used to improve robustness \cite{makarenkov2010weighted}.
This is a natural setting where \textit{weighted} Fermat-Weber points could be applied in future research.

The paper is organized as follows. In \cref{sec:background} we state our notation and build connections among the objects from tropical and polyhedral geometry relevant to our approach. In \cref{sec:main}, we prove our main result, \Cref{thm:main}, which is split into two parts: \Cref{thm:containment} and \Cref{thm:main-2}. Then in \cref{sec:conclusion}, we describe applications of weighted asymmetric tropical Fermat-Weber points to phylogenetics. We end by briefly outlining directions for future research.

\section{Preliminaries}
\label{sec:background}

In this section, our definitions from tropical geometry our sourced from \cite{ETC}; for polyhedral geometry our primary reference is \cite{triangulations-textbook}.

\subsection{The Weighted Asymmetric Tropical Fermat-Weber Problem}

The purpose of this subsection is to introduce the \textit{weighted asymmetric tropical Fermat-Weber problem}, and to formulate the weighted asymmetric tropical Fermat-Weber points as the minimizers of a \textit{tropical signomial} (see \Cref{prop:trop-sig}).

We will work in the tropical max-plus semi-ring $\mathbb{T}_\mathrm{max} = (\R \cup \{-\infty\}, \oplus, \odot)$, or simply $\mathbb{T}$ for brevity, where tropical addition $\oplus$ and tropical multiplication $\odot$ are defined by
\[
a \oplus b = \max\{a,b\} \quad \text{ and } \quad a \odot b = a + b, \quad \text{ for all} \ a,b \in \R \cup \{ -\infty \}.
\]
The tropical additive identity is $-\infty$ and the multiplicative identity is $0$. Note that the above operations extend to the semi-module $\mathbb{T}^n$ via coordinate-wise tropical addition and tropical scalar multiplication.

We consider data in the \textit{tropical projective torus}, which is the space $\TP$, where $\1_n$ denotes the vector of all ones in $\R^n$.
We consider $\TP$ with the following asymmetric tropical distance, which was first defined in \cite{joswig-com-tropical-medians}; see also \cite{allamigeon2018log}.

\begin{defn}[Comǎneci and Joswig {\cite[Section 3]{joswig-com-tropical-medians}}]
    \label{defn:asym-dist}
    The \textit{asymmetric tropical distance}, $d_\Delta(\bfx, \bfy)$ on $\TP$ is:
    \begin{equation}\label{eqn:asym-dist-defn}
        d_\Delta(\bfx,\bfy) := n \max_{i \in [n]} (x_i - y_i) + \sum_{i \in [n]} (y_i - x_i).
    \end{equation}
    It is possible to simplify the expression for $d_\Delta(\bfx, \bfy)$ by assuming that $\bfx$ and $\bfy$ are given by a specific representative in $\R^n$.
    Let $\HH$ be the subspace where the coordinates sum to zero:
    \[ \HH := \{ \bfx \in \R^n \mid x_1 + \cdots + x_n = 0 \}. \]
    Each point in the tropical projective torus has a unique representative in $\HH$, which is obtained via orthogonal projection,
    e.g.\, $(3, 1, 2) = 2 \1_3 + (1, -1, 0) \equiv (1, -1, 0)$.
    When the points $\bfx, \bfy \in \TP$ are given by their unique representative in $\HH$, the distance $d_\Delta(\bfx,\bfy)$ can be simplified to the following
\begin{equation}\label{eqn:asym-dist-H0-defn}
    d_\Delta(\bfx,\bfy) := n \max_{i \in [n]} (x_i - y_i), \phantom. \bfx, \bfy \in \HH.
\end{equation}
\end{defn}

From now on, we fix a data set $V = \{ \bfv_1, \ldots, \bfv_m \} \subset \HH$ and view this as a data set in $\TP$ where each point is given by its representative in $\HH$. 
Moreover, we allow the data points to be weighted by positive real weights $w_1, \ldots, w_m$.
We denote this list of weights by the vector $\bfw$.
In some instances, it will be essential that $\bfw$ is a \textit{weight vector}, meaning that in addition to assuming $w_i > 0$ for all $i$, we also assume that $\sum_{i=1}^m w_i = 1$.
In fact, we may always assume that $\bfw$ is a weight vector, since the minimizers of the sum in (\ref{eqn:fw-sum}) are unchanged by a uniform rescaling of the weights.

\begin{defn}\label{defn:FW}
    Given a data set $V = \{ \bfv_1, \ldots, \bfv_m \} \subset \HH$ and a weight vector $\bfw \in \R^m_{>0}$, the \textit{weighted (asymmetric tropical) Fermat-Weber set} on the data $V$ with weights $\bfw$ consists of the points $\bfx \in X$ minimizing the following sum
    \begin{equation}\label{eqn:fw-sum}
        \mbox{FW}(V, \bfw) := \underset{\bfx \in X}{\mathrm{argmin}} \frac1m \sum_{i=1}^m w_i d_\Delta(\bfx, \bfv_i).
    \end{equation}
    A \textit{weighted (asymmetric tropical) Fermat-Weber point} of $V$ with weights $\bfw$ is any point in the Fermat-Weber set $\mbox{FW}(V, \bfw)$.
\end{defn}

Note that \Cref{defn:FW} fixes a choice of signs by choosing the order of the arguments for $d_\Delta$. In particular, it fixes the tropical convex hull to be the min-tropical convex hull.

The next step is to rewrite the sum in (\ref{eqn:fw-sum}) using the operations of the tropical max-plus semi-ring.
For this, we require the following notation.

The \textit{tropical monomial} $\bfx^{\odot \bfa}$ is the sum $\sum_{i=1}^n a_i x_i$. A \textit{tropical signomial} is a finite tropical linear combination of tropical monomials with non-negative real exponents, i.e.\
$$f(\bfx) = \bigoplus_{\bfa \in A} \lambda_\bfa \odot \bfx^{\odot \bfa},$$
for a finite subset $A \subset \R^n_{\geq 0}$ is finite and $\lambda_\bfa \in \T$ for all $\bfa \in A$. If $A \subset \Z^n_{\geq 0}$, then $f(\bfx)$ is a \textit{tropical polynomial}.

Below, we reinterpret the distance $d_\Delta(\bfx, \bfv_i)$ as a tropical polynomial. The sum in (\ref{eqn:fw-sum}) is then a tropical signomial, which the Fermat-Weber points minimize.

\begin{defn}
    \label{defn:f-V-w}
    The \textit{tropical linear form centered at} $\bfv_i$ is defined to be
    \begin{equation*}
        f_{\bfv_i}(\bfx) := \bigoplus_{j=1}^n (-v_{ij}) \odot x_j,
    \end{equation*}
    and the \textit{tropical signomial associated to data $V$ with weights $\bfw$} is 
    \[ f_{V, \bfw}(\bfx) := \bigoplus_{i=1}^m f_{\bfv_i}^{\odot w_i}(\bfx). \] 
\end{defn}

\begin{prop}
    \label{prop:trop-sig}
    The weighted Fermat-Weber points of $V$ with respect to the tropical asymmetric metric $d_\Delta$ are the minimizers of $f_{V, \bfw}$ over $\HH$.
\end{prop}

\begin{proof}
The result follows from writing the classical sum in (\ref{eqn:fw-sum}) using the tropical operation $\odot = +$, and replacing $d_\Delta(\bfx, \bfv_i)$ with $f_{\bfv_i}$.
Indeed, the distance from $\bfx \in \HH$ to $\bfv_i \in \HH$ is (up to scaling by $1/n$) given by the tropical polynomial $f_{\bfv_i}(\bfx)$

\begin{equation*}
    \label{eqn:dist-to-vi}
    f_{\bfv_i}(\bfx) = \bigoplus_{j=1}^n (-v_{ij}) \odot x_j = \max_{j \in [n]} (x_j - v_{ij}) = \frac1n d_\Delta (\bfx, \bfv_i).
\end{equation*}

\noindent Thus, the average distance from $\bfx \in \HH$ to the data $V \subseteq \HH$ is (up to scaling by $m/n$) given by the following tropical signomial

\begin{equation*}
    \label{eqn:f-fw}
    f_{V,\bfw}(\bfx) = \bigodot_{i=1}^m f_{\bfv_i}^{\odot w_i} = 
    \sum_{i = 1}^m w_i \max_{j \in [n]}(x_j - v_{ij}) = \frac{1}{n} \sum_{i = 1}^m w_i d_\Delta(\bfx, \bfv_i).
\end{equation*}
It follows that the set of minimizers of $f_{V, \bfw}(\bfx)$ over $\bfx \in \HH$ is exactly $\FW(V, \bfw)$.
\end{proof}

We end the subsection with an example of $f_{V, \bfw}$ for two data points, which we will return to throughout the paper.

\begin{example}
    \label{ex:running-ex}
    Consider the data set $\{ \bfv_1 = (0,0,0), \bfv_2 = (1,-1,0) \} = V \subset \R^3/\R\1_3$.
    For the weights $\bfw = (\frac{1}{3},\frac{2}{3})$, $f_{V,\bfw}(\bfx)$ is the following tropical signomial
    \begin{align*}
        f_{V,\bfw}(\bfx) &= \left(x_1 \oplus x_2 \oplus x_3 \right)^{\nicefrac{1}{3}} \odot \left( -1 \odot x_1 \oplus 1 \odot x_2 \oplus x_3\right)^{\nicefrac{2}{3}}\\
        &= - \frac{2}{3} \odot x_1 \oplus 
        - \frac{2}{3} \odot x_1^{\nicefrac{2}{3}} x_2^{\nicefrac{1}{3}} \oplus 
        -\frac{2}{3} \odot x_1^{\nicefrac{2}{3}} x_3^{\nicefrac{1}{3}} \oplus 
        \frac{2}{3} \odot x_1^{\nicefrac{1}{3}}x_2^{\nicefrac{2}{3}} \\
        &\hspace{.5cm} \oplus \frac{2}{3} \odot x_2
        \oplus \frac{2}{3} \odot x_2^{\nicefrac{2}{3}}x_3^{\nicefrac{1}{3}} \oplus 
        x_1^{\nicefrac{1}{3}}x_3^{\nicefrac{2}{3}} \oplus 
        x_2^{\nicefrac{1}{3}}x_3^{\nicefrac{2}{3}} \oplus 
        x_3 \\
        &= - \frac{2}{3} \odot x_1 \oplus \frac{2}{3} \odot x_1^{\nicefrac{1}{3}} x_2^{\nicefrac{2}{3}} \oplus x_1^{\nicefrac{2}{3}} x_3^{\nicefrac{1}{3}} \oplus \frac{2}{3} \odot x_2 \oplus \frac{2}{3} \odot x_2^{\nicefrac{2}{3}}x_3^{\nicefrac{2}{3}} \oplus x_3.
    \end{align*}
    We can simplify $f_{V, \bfw}$ to a tropical signomial with six terms since three of the nine original terms never achieve the max in $f_{V, \bfw}$. Thus, the above function is piecewise linear with 6 linear pieces of maximal dimension.
\end{example}

\subsection{Minimizing Tropical Signomials}

In this subsection, we describe the Fermat-Weber points of data $V \subset \TP$ with weights $\bfw \in \R^m_{>0}$ using convex and polyhedral geometry.
This connection is facilitated by the formulation of the Fermat-Weber points in the previous section as the minimizers of the tropical signomial $f_{V, \bfw}$ over $\HH$ (\Cref{prop:trop-sig}).
Note that while the results we cite from \cite{ETC} are stated for tropical \textit{polynomials}, the proofs easily extend to tropical signomials.

Fix a tropical max-plus signomial, $f(\bfx) = \bigoplus_{\bfa \in A} c_\bfa \odot \bfx^{\odot \bfa}$. 
It is a piecewise linear, continuous, convex function on $\R^n$. 
We are interested in minimizing $f$ over the subspace $\HH$, which we previously identified with $\TP$.
This minimum, if it exists, is achieved on a face of $\epi(f) \cap \left( \HH \times \R \right)$, where $\epi(f)$ is the \textit{epigraph}\footnote{Note that \cite{ETC} refers to the \textit{dome} of $f$ instead of the epigraph of $f$ because the min-tropical semi-ring is used in that source. The epigraph is the equivalent object for the max-tropical convention.} of $f$
\[ \epi(f) := \{ (\bfx, t) \in \R^n \times \R \mid \bfx \in \R^n, t \geq f(\bfx) \}. \]
The epigraph is an unbounded polyhedron of dimension $n + 1$ \cite[Proposition 1.1]{ETC}.
Each facet of $\epi(f)$ is supported by the hyperplane $\{ (\bfx, t) \mid t = c_{\bfa_F} + \bfa_F \cdot \bfx \}$ for some $\bfa_F \in A$.
Projecting away the last coordinate of the epigraph gives the normal complex of $f$, whose cells are the regions of linearity of $f$.

\begin{defn}[{cf.\ \cite[p. 6]{ETC}}]
    \label{defn:normal-complex}
    Let $f$ be a tropical signomial. The \textit{normal complex} of $f$, which we denote $\NC(f)$, is the polyhedral complex $\{ \pi(F) \mid F \text{ is a face of } \epi(f) \}$, where $\pi: \R^n \times \R \to \R^n$ is the projection to the first factor.
\end{defn}

We will further assume that $f$ is \textit{homogeneous}, i.e.\ the sum of the coordinates of $\bfa$ is the same for all $\bfa \in A$.
For example, $f_{\bfv_i}$ and $f_{V, \bfw}$ are homogeneous of degrees one and $\sum_{i=1}^m w_i$ respectively.
For homogeneous $f$, the lineality space of $\NC(f)$ contains $\1_n$, and thus there is a bijection between cells of $\NC(f)$ and cells of the image $\overline{\NC}(f)$ in $\TP$.
Again, we make the identification $\TP \cong \HH$ and note that with this identification $\overline{\NC}(f) = \NC(f) \cap \HH$.

We now pass to the dual picture.
The \textit{support} of $f$, which we denote by $\supp(f)$, is the set of $\bfa \in A$ such that the coefficient $c_\bfa$ is finite; the \textit{Newton polytope} of $f$ is the convex hull of $\supp(f)$.
For example, the Newton polytope of $f_{\bfv_i}$ is the standard simplex $\Delta^{n-1} = \conv(\bfe_1, \ldots, \bfe_n)$.
The \textit{subdivision dual to} $f$, denoted $\newtsubdiv(f)$, is the regular subdivision\footnote{We use the max-tropical convention, so regular subdivisions are defined via the \textit{upper} convex hull.} of $\supp(f)$ induced by lifting the point $\bfa$ to $(\bfa, c_\bfa)$ for all $\bfa \in \supp(f)$.
\Cref{fig:tropVf} illustrates this subdivision, the epigraph, and the normal complex for the tropical signomial in \Cref{ex:running-ex}.

\begin{figure}[h]
    \centering
    \begin{minipage}{.5\textwidth}
        \centering
        \includegraphics[scale=.85]{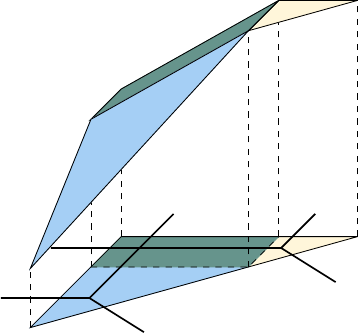}
    \end{minipage}%
    \begin{minipage}{.5\textwidth}
        \centering
        \includegraphics[scale=.95]{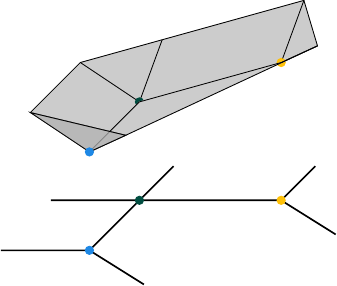}
    \end{minipage}
    \caption{Left: The regular subdivision $\newtsubdiv(f)$ for the tropical polynomial
    $f = - \frac{2}{3} \odot x_1 \oplus \frac{2}{3} \odot x_1^{\nicefrac{1}{3}} x_2^{\nicefrac{2}{3}} \oplus x_1^{\nicefrac{2}{3}} x_3^{\nicefrac{1}{3}} \oplus \frac{2}{3} \odot x_2 \oplus \frac{2}{3} \odot x_2^{\nicefrac{2}{3}}x_3^{\nicefrac{2}{3}} \oplus x_3$. Overlaid is the normal complex $\overline{\NC}(f_{V, \bfw})$. Right: $\epi(f_{V, \bfw}) \cap \HH$ over $\overline{\NC}(f_{V, \bfw})$.}
    \label{fig:tropVf}
\end{figure}

The cells of $\newtsubdiv(f)$ are in bijection with the faces of $\epi(f)$ via the following inclusion reversing correspondence $\beta$ (see \cite[Theorem 1.13]{ETC}).
Each facet $F$ of $\epi(f)$ arises from exactly one term $c_{\bfa_F} \odot \bfx^{\odot \bfa_F}$ of $f$.
Thus, $\beta(F) = \{ \bfa_F \}$.
Now $\beta$ extends to any face by sending the intersection of facets $F_1, \ldots, F_k$ to the convex hull of $\beta(F_1), \ldots, \beta(F_k)$.
The polyhedral complex $\{ \beta(F) \mid F \text{ is a face of } \epi(f) \}$ is exactly $\newtsubdiv(f)$, and each cell $\beta(F)$ encodes the supporting hyperplanes of $F$.
The \textit{cell dual to} $C \in \overline{\NC}(f)$ is $\beta(F)$ where $F$ is the face of $\epi(f)$ such that $C = \overline{\pi(F)}$.
This is well-defined because there is a bijection between cells of $\overline{\NC}(f)$ and faces of $\epi(f)$.

We will say a cell of the subdivision $\newtsubdiv(f)$ is \textit{interior} if it is not contained in a facet of the Newton polytope of $f$; the relative interior of a polytope $P$ is denoted $\relint(P)$.
Interior cells correspond under $\beta$ to the \textit{bounded} cells of $\overline{\NC}(f)$, i.e. the cells which are contained in a classical (Euclidean) ball of finite radius when $\overline{\NC}(f)$ is viewed as a polyhedral complex in $\HH \cong \R^{n-1}$.
The cell of $\newtsubdiv(f)$ that contains $\lambda \1_n$ for some $\lambda \in \R$ in its relative interior is the \textit{central cell} of the Newton polytope.

The dual correspondence $\beta$ leads to the following description of the minimizers of $f$ over~$\HH$.

\begin{lemma}
    \label{lem:barycenter}
    Let $f$ be a homogeneous tropical signomial. The set of minimizers of $f$ over $\HH$ is the cell of $\overline{\NC}(f)$ dual to the central cell of $\newtsubdiv(f)$ for some $\lambda \in \R$. The minimum exists if and only if such a face exists.
\end{lemma}

\begin{proof}
    The minimum of $f$ over $\HH$ exists if and only if there is a face $\widehat{F}$ of $\epi(f)$ supported by a hyperplane of the form $t = \bfm \cdot \bfx + c$, where $\langle \bfm, \bfx \rangle = 0$ when restricted to $\HH$.
    Such a hyperplane must have $\bfm = \lambda \1_n$ for $\lambda \in \R$.
    In terms of the dual subdivision, this means $\widehat{F}$ is a minimizing face if and only if $\lambda \1_n \in \relint(\beta(\widehat{F}))$. The set of minimizers of $f$ over $\HH$ is then the cell of $\overline{\NC}(f)$ that arises as the projection of $\widehat{F}$ to $\HH$.
\end{proof}

We take a brief detour to describe the combinatorics of the subdivision dual to a product of tropical polynomials with positive real exponents, e.g.\ $f_{V, \bfw}$.
First, we recall the situation for products of tropical polynomials.
Given a tropical polynomial $f$ along with a factorization into tropical polynomials, $f(\bfx) = \bigodot_{i=1}^m f_i(\bfx)$, the subdivision dual to $f$ is a regular \textit{mixed subdivision} of the Minkowski sum $N_1 + \cdots + N_m$, where $N_i = \supp(f_i)$ (see \cite[Section 4.1, 4.2]{ETC}).
The (regular) \textit{mixed subdivisions} of $N_1 + \cdots + N_m$ are exactly those subdivisions arising from a (regular) subdivision of the \textit{Cayley embedding} of $N_1, \ldots, N_m$, which is 
\[\mathrm{Cayley}(N_1, \ldots, N_m) := \bfe_1 \times N_1 \cup \cdots \cup \bfe_m \times N_m \subset \R^m \times \R^n.\]

The scaled Minkowski sum $\frac{1}{m} (N_1 + \cdots + N_m)$ is identified with the linear section of the Cayley polytope $W(\frac{1}{m} \1_m) \cap \mathrm{Cayley}(N_1, \ldots, N_1)$, where $W(\bfw) := \{ \bfw \} \times \R^n$ for any weight vector $\bfw$.
A subdivision $\Sigma$ of the Cayley polytope gives rise to the subdivision $W \cap \Sigma := \{ W(\frac{1}{m}) \cap C \mid C \in \Sigma \}$ of the linear section.
Then 
\[\mathrm{M}(\Sigma) := \{ \conv(A_1 + \cdots + A_m) \mid \conv(\bfe_1 \times A_1, \ldots, \bfe_m \times A_m) \in \Sigma \} \]
is the corresponding subdivision of $N_1 + \cdots + N_m$.
We say a subdivision of $N_1 + \cdots + N_m$ that arises from this construction \textit{coincides} with $W \cap \Sigma$.
The correspondence is known as the Cayley trick. It was first proved in \cite{sturmfels-cayley-trick} for regular subdivision, and then in \cite{hrs-cayley-trick} for all subdivisions; see also \cite[Section 9.2]{triangulations-textbook}.
The specific regular subdivision of $\mathrm{Cayley}(N_1, \ldots, N_m)$ that gives rise to the subdivision dual to $f$ is the following.

\begin{defn}[{see \cite[Corollary 4.9]{ETC}}]
    \label{defn:cayley-subdivsion}
    Given tropical polynomials $f_i(\bfx) = \bigoplus_{\bfa \in A_i} c_{i, \bfa} \odot \bfx^{\odot \bfa}$, $i = 1, \ldots, m$, the \textit{Cayley subdivision dual to} $f(\bfx) = \bigodot_{i=1}^m f_i(\bfx)$, denoted $\Sigma(f_1, \ldots, f_m)$, is the regular subdivision of $\mathrm{Cayley}(\supp(f_1), \ldots, \supp(f_m))$ induced by lifting $(\bfe_i, \bfa)$ to $(\bfe_i, \bfa, c_{i, \bfa})$.
\end{defn}

The Cayley trick generalizes to the weighted case, where the weighted Minkowski sum $w_1 N_1 + \cdots + w_m N_m$ is identified with the linear section $W(\bfw) \cap \mathrm{Cayley}(N_1, \ldots, N_m)$ (for more details, see \cite[Section 3]{hrs-cayley-trick}).
Combining this with \cite[Corollary 4.9]{ETC} leads to the following statement, which describes the subdivision dual to a product of tropical polynomials with positive real exponents.

\begin{lemma}
    \label{thm:cayley-trick}
    Let $f_1, \ldots, f_m$ be tropical polynomials and let $\bfw \in \R^m_{>0}$ be a weight vector. Recall that the entries of a weight vector sum to one. The subdivision dual to $f(\bfx) = \bigodot_{i=1}^m f_i (\bfx)$ coincides with $W(\bfw) \cap \Sigma(f_1, \ldots, f_m)$.
\end{lemma}

\begin{proof}
    The equally weighted case, where $w_i = \frac{1}{m}$ for all $i$, is Corollary 4.9 of \cite{ETC}.
    Varying the weights shifts the linear section of the Cayley polytope from $W(\frac{1}{m} \1_m ) \cap \Sigma(f_1, \ldots, f_m)$ to $W (\bfw) \cap \Sigma(f_1, \ldots, f_m)$ (see \cite[Lemma 3.2]{hrs-cayley-trick}).
\end{proof}

We conclude the subsection with an illustration of the Cayley trick for the data set from \Cref{ex:running-ex}.

\begin{example}
    Let $V = \{ \bfv_1 = (0,0,0), \bfv_2 = (1,-1,0) \}$. The Cayley subdivision dual to $f_{\bfv_1}^{w_1} \odot f_{\bfv_2}^{w_2}$ is the subdivision of $\Delta^{1} \times \Delta^2$ illustrated in \Cref{fig:cayley-trick}. Note that it does not depend on $w_1, w_2$.
    \Cref{fig:cayley-trick} also contains the subdivision (of the support) dual to $f_{\bfv_1}^{\nicefrac{1}{2}} \odot f_{\bfv_2}^{\nicefrac{1}{2}}$ (the unweighted case) and the subdivision (of the support) dual to $f_{\bfv_1}^{\nicefrac{1}{3}} \odot f_{\bfv_2}^{\nicefrac{2}{3}}$.

    \begin{figure}[h]
        \centering
        \includegraphics{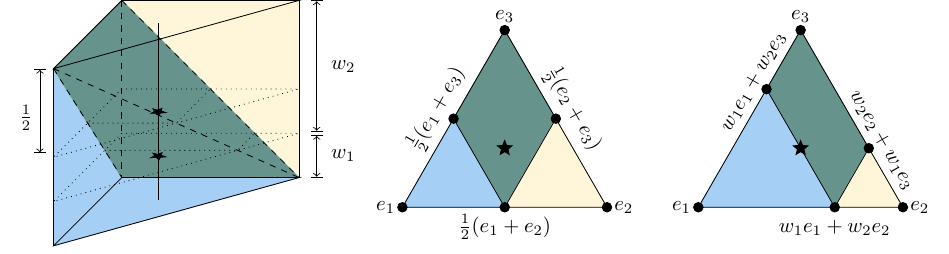}
        \caption{Left to right: the Cayley subdivision for the data $V = \{ \bfv_1 = (0,0,0), \bfv_2 = (1,-1,0) \}$, the subdivision (of the support) dual to $f_{V, \frac{1}{2} \1_2}$, and the subdivision (of the support) dual to $f_{V, (\frac{1}{3}, \frac{2}{3})}$.
        In the Cayley subdivision, the top star is the point $(\frac{1}{2} \1_2, \frac{1}{3} \1_3)$ and the bottom star is the point $(\bfw, \frac{1}{3} \1_3)$. 
        In the middle and right pictures, the star is the point $\frac{1}{2} \1_3$.}
        \label{fig:cayley-trick}
    \end{figure}
\end{example}

\subsection{Tropical Convexity}

In this subsection, we connect the Fermat-Weber points to the tropical convex hull of the data. 
We begin by discussing the tropical vanishing set of $f_{V, \bfw}$.

\begin{defn}[{cf.\ \cite[Definition 1.3]{ETC}}]
    \label{defn:tropical-vanishing}
    The \textit{tropical vanishing set} or \textit{tropical hypersurface} of a tropical signomial $f(\bfx) = \bigoplus_{\bfa \in A} c_\bfa \odot \bfx^{\bfa}$, denoted $\trV(f)$, is the set of $\bfx \in \R^n$ for which the maximum in $f(\bfx)$ is achieved at least twice.
\end{defn}

We continue to assume that $f$ is homogeneous, and so $\trV(f)$ contains $\1_n$ in its lineality space.
The image of $\trV(f)$ in $\TP$ is denoted by $\overline{\trV}(f)$.
In terms of the previous subsection, the tropical hypersurface $\trV(f)$ is the codimension-1-skeleton of the normal complex\footnote{Or, equivalently, the codimension-2-skeleton of the epigraph of $f$, projected onto the first factor.} of $f$, $\NC(f)$, and similarly for $\overline{\trV}(f)$ and $\overline{\NC}(f)$ \cite[Corollary 1.6]{ETC}.
Thus, $\overline{\trV}(f)$ has the structure of a polyhedral complex.

It is well-known that the vanishing set of a product of tropical polynomials is the union of the tropical vanishing sets of each factor \cite[Lemma 4.6]{ETC}.
The following result is the analogue for tropical signomials.

\begin{lemma}
    \label{prop:tropical-vanishing-product}
    Let $f$ be a tropical polynomial, and let $w > 0$. Then $\trV(f^w)$ is (as a set) equal to $\trV(f)$. Moreover,
    $$\trV\left( \bigodot_{i=1}^m f_i^{\odot w_i} \right) = \bigcup_{i=1}^m \trV(f_i), \text{ for any } \bfw \in \R_{> 0}^m.$$
    In particular, $\trV(f_V) = \trV(f_{V, \bfw})$.
\end{lemma}

\begin{proof}
    Let $f_i = \bigoplus_{\bfa \in A_i} c_\bfa \odot \bfx^{\odot \bfa}$, $A_i \subset \R^n_{\geq 0}$. 
    If the maximum in $f_j$ is achieved twice by $\bfx$, then $f_j^{\odot w_j}(\bfx) = w_j(c_{\bfa_1} + \bfa_1 \cdot \bfx) = w_j(c_{\bfa_2} + \bfa_2 \cdot \bfx)$ for some $\bfa_1 \neq \bfa_2 \in A_j$. 
    It follows that the maximum in $\bigodot_{i=1}^m f_i^{\odot w_i}$ is also achieved at least twice: 
    $$\bigodot_{i=1}^m f_i^{\odot w_i} = w_j(c_{\bfa_1} + \bfa_1 \cdot \bfx) + \sum_{i \neq j} f_i^{\odot w_i}(\bfx) = w_j(c_{\bfa_2} + \bfa_2 \cdot \bfx) + \sum_{i \neq j} f_i^{\odot w_i}(\bfx).$$
    On the other hand, if the maximum is achieved twice in $\bigodot_{i=1}^m f_i^{\odot w_i}$ at $\bfx$, then we must be able to write the maximum as two distinct sums: $\sum_{i=1}^m w_i (c_{\bfa_i^1} + \bfa_i^1 \cdot \bfx) = \sum_{i=1}^m w_i (c_{\bfa_i^2} + \bfa_i^2 \cdot \bfx)$, where $\bfa_i^1, \bfa_i^2 \in A_i$. It follows that for some $j$, $\bfa_j^1 \neq \bfa_j^2$, so the maximum in $f_j^{\odot w_j}$ is achieved at least twice.
\end{proof}

The tropical vanishing set of $f_{\bfv_i}$ is the \textit{tropical hyperplane centered at $\bfv_i$}, and 
thus, $\trV(f_V)$ is the \textit{arrangement of tropical hyperplanes centered at} $\bfv_1,\ldots,\bfv_m \in V$.
Since $\trV(f_{V, \bfw})$ is independent of the specific weights $\bfw$, it is also the tropical hyperplane arrangement centered at $V$.
A tropical hyperplane arrangement induces a polyhedral subdivision of $\R^n$ (and of $\TP$) called the \textit{(tropical) covector decomposition induced by $V$}, which was introduced by Develin and Sturmfels in \cite{develin2004tropical}; cells of this subdivision are called \textit{covector cells}.
We use $\CD(V)$ to denote the covector subdivision of $\TP$ by $V$.
The covector decomposition agrees with the normal complex $\overline{\NC}(f_{V, \1_m})$ (see \cite[Section 6.3]{ETC}), and below, we show that it also agrees with the normal complex of $f_{V, \bfw}$ for any $\bfw \in \R^m_{>0}$.

\begin{lemma}
    \label{cor:normal-complex-fVw}
    For any $\bfw \in \R^m_{>0}$, we have $\NC(f_{V, \1_m}) = \NC(f_{V, \bfw})$.
    Moreover, the covector decomposition induced by $V$ agrees with $\overline{\NC}(f_{V, \bfw})$ for any $\bfw \in \R^m_{> 0}$.
\end{lemma}

\begin{proof}
    The tropical vanishing set $\trV(f)$ determines the normal complex $\NC(f)$, so by \Cref{prop:tropical-vanishing-product}, $\NC(f_{V, \1_m}) = \NC(f_{V, \bfw})$ for all $\bfw \in \R^m_{>0}$.
    The subdivision $\CD(V)$ is the normal complex $\overline{\NC}(f_{V, \1_m})$ \cite[Section 6.3]{ETC}, and thus $\CD(V)$ agrees with $\overline{\NC}(f_{V, \bfw})$.
\end{proof}

The union of all bounded (max-tropical) covector cells of $V$ forms the \textit{min-tropical convex hull} of $V$ \cite[Observation 6.12]{ETC}, the definition of which we recall below.
In particular, the \textit{min}-tropical convex hull of $V$ has the structure of a polyhedral complex whose cells are exactly the bounded (max-tropical) covector cells of $V$.
The \textit{min}-tropical convex hull uses the operations of the \textit{min}-tropical semi-ring, $a \boxplus b = \min(a, b)$ and $a \odot b = +$.

\begin{defn}[{see \cite[Section 5.2]{ETC}}]
    \label{defn:tconv}
    The \textit{min-tropical convex hull} of $\{ \bfv_1, \ldots, \bfv_m \} = V \subset \TP$, which we denote by $\tconv(V)$ or $\tconv(\bfv_1, \ldots, \bfv_m)$, is the set of all $\min$-tropical linear combinations of points in $V$. More precisely,
    \begin{equation}
        \label{eqn:tconv-defn}
        \tconv(V) := \{ \lambda_1 \odot \tilde{\bfv}_1 \boxplus \cdots \boxplus \lambda_m \odot \tilde{\bfv}_m \mid \lambda_i \in \R \},
    \end{equation}
    where $\tilde{\bfv}_i \in \R^n$ is any representative of $\bfv_i$.
\end{defn}

Note that tropical convex hull is independent of the representatives $\tilde{\bfv}_i$ chosen for the points $\bfv_i$ in $V$. That is, if $\tilde{\bfv}_i^\prime = \mu_i \odot \tilde{\bfv}_i$ for some $\mu_i \in \R$, then
\begin{equation*}
    \label{eqn:diff-reps-equiv}
    \lambda_1 \odot \tilde{\bfv}_1 \boxplus \cdots \boxplus \lambda_m \odot \tilde{\bfv}_m = (\lambda_1 - \mu_1) \odot \tilde{\bfv}_1^\prime \boxplus \cdots \boxplus (\lambda_m - \mu_m) \odot \tilde{\bfv}_m^\prime.
\end{equation*}

\begin{example}
    \label{ex:tconv}
    Let $\bfv_1 = (0,0,0)$, and $\bfv_2 = (1,-1,0)$. The tropical polytope $\tconv(\bfv_1, \bfv_2)$ is illustrated in \Cref{fig:tconv-ex}. It consists of five covector cells: three points and two line segments.

    \begin{figure}[h]
        \centering
        \includegraphics{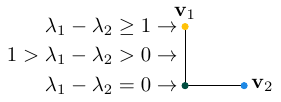}
        \caption{The tropical line segment $\tconv((0,0,0), (1,-1,0))$.}
        \label{fig:tconv-ex}
    \end{figure}
\end{example}

We conclude the section by recalling \textit{covector graphs}, which were introduced by Develin and Sturmfels in \cite{develin2004tropical}.
In terms of the notation in this paper, we can understand the correspondence between covector cells and bipartite graphs as follows.
The covector decomposition $\CD(V) = \overline{\NC}(f_{V, \bfw})$ is dual the mixed subdivision of the Newton polytope $\newtsubdiv({f_{V, \bfw}})$, which arises via the Cayley trick to a subdivision of $\mathrm{Cayley}(\supp(f_{\bfv_1}), \ldots, \supp(f_{\bfv_m}))$.
As the Newton polytope of each factor $f_{\bfv_i}$ is an $(n-1)$ dimensional simplex, the Cayley embedding $\mathrm{Cayley}(\supp(f_{\bfv_1}), \ldots, \supp(f_{\bfv_m}))$ is the product of simplices $\Delta^{m-1} \times \Delta^{n-1}$.
Thus, each covector cell corresponds to a cell in a subdivision of~$\Delta^{m-1} \times \Delta^{n-1}$.

There is a well-known correspondence between subsets of the vertices of $\Delta^{m-1} \times \Delta^{n-1}$ and subgraphs of the complete bipartite graph with $m$ left vertices, and $n$ right vertices, which we call $K_{m,n}$ (for more details, see \cite[Section 6.2.2]{triangulations-textbook}). 
The vertex $(\bfe_i, \bfe_j) \in \Delta^{m-1} \times \Delta^{n-1}$ corresponds to the edge between left vertex $i$ and right vertex $j$ in the bipartite graph. 
Thus, a subset of vertices $A$ of the vertices of $\Delta^{m-1} \times \Delta^{n-1}$ corresponds to the subgraph of $K_{m,n}$.
See \Cref{ex:covector-graphs} for covector graphs of the running example.

\begin{example}[Simplex-Forest Correspondence for $m = n = 2$]
    \label{ex:covector-graphs}
    The product of two $2$-simplices (i.e.\ line segments) is a square; the corresponding bipartite graph has two left vertices and two right vertices. Both are illustrated in \Cref{fig:forest-simplex}. The vertex $(\bfe_i, \bfe_j)$ in the simplex corresponds to the edge $(l_i, r_j)$ in the bipartite graph. For example, the top left vertex of the shaded gray simplex, $(\bfe_1, \bfe_2)$, corresponds to the edge $(l_1, r_2)$.
    \begin{figure}[!h]
        \centering
        \includegraphics{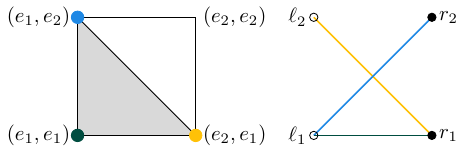}
        \caption{Left: vertices in the product of simplices. Right: the corresponding edges of the bipartite subgraph of $K_{2,2}$.}
        \label{fig:forest-simplex}
    \end{figure}
\end{example}

\section{The weighted tropical Fermat-Weber problem}
\label{sec:main}

We now come to the proof of \Cref{thm:main}, which is broken into two parts: \Cref{thm:containment} and \Cref{thm:main-2}.
First, we prove that the weighted Fermat-Weber points form a covector cell of the min-tropical convex hull of the data.

\begin{theorem}
    \label{thm:containment}
    Given data $\{ \bfv_1, \ldots, \bfv_m \} = V \subset \TP$, and a weight vector $\bfw \in \R^m$, the weighted (tropical asymmetric) Fermat-Weber set is the covector cell of $\mbox{tconv}(\bfv_1, \ldots, \bfv_m)$ that is dual to central cell of $\newtsubdiv(f)$.
\end{theorem}

\begin{proof}
    We will work backwards, starting with the Newton polytope.
    The Newton polytope is Minkowski additive over products, so the Newton polytope of $f_{V, \bfw}$ is $w_1 \Delta^{n-1} + \cdots + w_m \Delta^{n-1}$.
    The point $\frac{1}{n} \1_n$ is contained in its relative interior.
    Therefore, there is an interior cell of $\newtsubdiv(f)$, denoted $C^\ast$, containing $\frac{1}{n} \1_n$ in its relative interior.
    Dual to $C^\ast$ is a cell of $\overline{\NC}(f_{V, \bfw})$, denoted $C$.
    The cell $C$ is bounded since $C^\ast$ is an interior cell of $\newtsubdiv(f)$.
    Thus, $C$ is a cell of the min-tropical convex hull of the data.
    According to \Cref{lem:barycenter}, $C$ is the set of minimizers of $f_{V, \bfw}$ over $\HH$, which by \Cref{prop:trop-sig} is $\FW(V, \bfw)$.
\end{proof}

The second half of \Cref{thm:main} says that for any covector cell $C$ of $\tconv(\bfv_1,\ldots,\bfv_m)$, there is some weight vector that makes $C$ the Fermat-Weber set.
To prove this, it is convenient to pass to the Cayley subdivision, and for this we introduce the following notation.
We denote Cayley subdivision dual to $f_{V, \bfw}$ by $\Sigma(V)$, since it does not depend on the weights.
It is the regular subdivision of the product of simplices $\Delta^{m-1} \times \Delta^{n-1}$ induced by lifting $(\bfe_i, \bfe_j)$ to $(\bfe_i, \bfe_j, -v_{ij})$.
The weights are accounted for by taking the appropriate linear section.

The central cell of the subdivision dual to $f_{V, \bfw}$ is the cell dual to the Fermat-Weber set (\Cref{prop:trop-sig}, \Cref{lem:barycenter}).
The Newton polytope of $f_{V, \bfw}$ is the simplex $\Delta^{n-1}$, and thus the central cell is the one containing $\frac{1}{n} \1_n$ in its relative interior. 
The \textit{central Cayley cell} is then defined as the cell of $\Sigma(V)$ which contains $(\bfw, \frac{1}{n} \1_n)$.
\Cref{thm:simplex-weights} proves that for any interior cell $C$ of \textit{any} subdvision of $\Delta^{m-1} \times \Delta^{n-1}$, there exists a weight vector so that $C$ contains $(\bfw, \frac{1}{n} \1_n)$. 

The proof of \Cref{thm:simplex-weights} formulates this containment as a fractional matching problem over a spanning forest of $K_{m,n}$ using the correspondence between subsets of vertices of $\Delta^{m-1} \times \Delta^{n-1}$ and bipartite graphs.
Fractional matchings have been extensively studied in combinatorial optimization; for further details see Schrijver's monograph \cite{schrijver2003combinatorial}.

\begin{prop}
    \label{thm:simplex-weights}
    Given any subdivision of the vertices of $\Delta^{m-1} \times \Delta^{n-1}$ and any interior cell $C$ of the subdivision, there is a choice of positive real weights $w_1, \ldots, w_m$ with $\sum w_i = 1$, so that $C$ contains the point $(w_1, \ldots, w_m, \frac1n \mathbbm{1}_n)$ in its relative interior.
\end{prop}

\begin{proof}
    Begin by refining the given subdivision to a triangulation, and let $S$ be a cell of that triangulation which is contained in $C$ and which has the same dimension as $C$.
    Then $S$ is an interior cell of the triangulation of $\Delta^{m-1} \times \Delta^{n-1}$, and if the point $(\bfw, \frac{1}{n} \1_n)$ lies in the relative interior of $S$, then it also lies in the relative interior of $C$.

    Let $F$ be the bipartite graph corresponding to the vertices of $S$. 
    By \cite[Lemma 6.2.8(a)]{triangulations-textbook}, $F$ is a forest in $K_{m,n}$; it is also spanning, which can be seen as follows.
    The intersection $S \cap \relint(\Delta^{m-1} \times \Delta^{n-1})$ is non-empty if and only if $F$ is not contained in the bipartite graph corresponding to any facet of $\Delta^{m-1} \times \Delta^{n-1}$.
    The bipartite graph corresponding to a facet of $\Delta^{m-1} \times \Delta^{n-1}$ is a subgraph of $K_{n,m}$ obtained by deleting a single vertex (see \cite[Section 6.2.2]{triangulations-textbook}).
    $S$ is an interior cell of the triangulation, hence not contained in any facet. 
    Thus $F$ spans $K_{n,m}$.

    Now, a point $(\bfp, \bfq) \in \R^m \times \R^n$ lies in $\relint(S)$ if it can be written as a convex combination of the vertices of $S$ with all positive coefficients.
    In terms of the forest $F$, $(\bfp, \bfq)$ lies in $\relint(S)$ if there exist $\lambda(e) > 0$ for each edge $e$ in $F$ such that the sum of edge weights on any left vertex adds up to the corresponding $\bfp$ coordinate, and the sum of edge weights on any right vertex adds up to the corresponding $\bfq$ coordinate. Let $r(e)$ be the node on the right side connected to $e$, and let $\ell(e)$ be the node on the left side connected to $e$. The choice of $\lambda$'s in (\ref{eqn:lambda-defn}) leads to a valid choice of weights $w_1, \ldots, w_m$ (given in (\ref{eqn:w-defn})) so that $S$ contains the weighted barycenter.
\begin{equation}\label{eqn:lambda-defn}
    \lambda(e) := \frac{1}{n \cdot \deg r(e)}.
\end{equation}
\begin{equation}\label{eqn:w-defn}
    w_i := \sum_{e \text{ s.t. } \ell(e) = i} \lambda(e).
\end{equation}
The equations in (\ref{eqn:right-sums}) show that the weights on any right node sum to $\frac1n$ (since $F$ is spanning, every vertex has at least one edge); by definition, the weights on the $i$th left node sum to $w_i$. It follows that $\bfb = (w_1, \ldots, w_m, \frac1n \mathbbm{1})$ lies in the relative interior of $S$.
\begin{equation}\label{eqn:right-sums}
    \sum_{r(e) = j} \frac{1}{n \cdot \deg(j)} = \frac1n \sum_{r(e) = j} \frac{1}{\deg(j)} = \frac1n \deg(j) \frac{1}{\deg(j)} = \frac1n.
\end{equation}

Moreover, $w_1, \ldots, w_m$ is a valid choice of weights for the Fermat-Weber problem. The weights $w_i$ are positive because $F$ is spanning, so the sum in (\ref{eqn:w-defn}) is never empty; the equations in (\ref{eqn:w-sum}) show that the $w_i$ sum to one.
\begin{equation}
    \label{eqn:w-sum}
    \sum_{i=1}^m w_i = \sum_e \lambda(e) = \sum_{j=1}^n \sum_{r(e) = j} \frac{1}{n \cdot \deg(j)} = \sum_{j=1}^n \frac1n \sum_{r(e) = j} \frac{1}{\deg(j)} = \sum_{j=1}^n \frac1n = n \frac{1}{n} = 1. \qedhere
\end{equation}
\end{proof}

We can now complete the proof of \Cref{thm:main}.

\begin{theorem}
    \label{thm:main-2}
    Let $P$ be any cell of the covector decomposition of the tropical polytope $\tconv(\bfv_1,\ldots,\bfv_m)$. There exists a weight vector, $\bfw_P$, so that $\FW(V, \bfw_P) = P$.
\end{theorem}

\begin{proof}
    Let $C$ be the cell corresponding to $P$ in the subdivision of $\Delta^{m-1} \times \Delta^{n-1}$ induced by lifting $(\bfe_i, \bfe_j)$ to $(\bfe_i, \bfe_j, -v_{ij})$.
    Since $P$ is bounded, $C$ is an interior cell, and thus \Cref{thm:simplex-weights} implies there is a weight vector $\bfw_P$ so that $(\bfw_P, \frac{1}{n} \1_n) \in \relint(C)$.
    Now let $\mathrm{M}(C)$ be the cell dual to $P$ in the subdivision dual to $f_{V, \bfw_P}$.
    The cell $\mathrm{M}(C)$ arises via the Cayley trick by projecting $W(\bfw_P) \cap C \subset \R^m \times \R^n$ to the second factor.
    Shifting the linear section has the effect of rescaling the simplices in the Minkowski sum $w_1 \Delta^{n-1} + \cdots + w_m \Delta^{n-1}$, and now $\frac{1}{n} \1_n \in \mathrm{M}(C)$.
    Now by \Cref{lem:barycenter} $P$ is the set of minimizers for $f_{V,\bfw_P}$ over $\HH$.
    Finally, \Cref{prop:trop-sig} implies that $P = \FW(V, \bfw_P)$.
\end{proof}

We conclude the section by remarking that the Fermat-Weber set $\FW(V, \bfw)$ can be explicitly recovered by computing the central cell of the Cayley polytope subdivision, which is the cell containing $(\bfw, \frac{1}{n} \1_n)$ in its relative interior.
Indeed, given the central cell $C$ of the Cayley polytope subdivision we can compute the central cell $\mathrm{M}(C)$ of the subdivision dual to $f$ as the projection of $W(\bfw) \cap C$ onto the second factor. 
It is a (weighted) mixed cell, so we have $\mathrm{M}(C) = w_1 A_1 + \cdots + w_m A_m$, where $C = \conv(\bfe_1 \times A_1 \cup \cdots \cup \bfe_m \times A_m)$. 
Then $\bfx \in \FW(V, \bfw)$ if and only if the terms of $f_{\bfv_i}$ with exponents $\bfa \in A_i$ simultaneously achieve the maximum in $f_{\bfv_i}(\bfx)$, for all $i$.
Note that for generic weights, the Fermat-Weber set is a point (of minimal dimension) since the central Cayley cell is of maximal dimension for generic weights.

\section{Applications to Phylogenetics}
\label{sec:conclusion}

In this section, we outline how weighted tropical Fermat-Weber points and \Cref{thm:main} can be applied to phylogenetics.
Phylogenetic reconstruction is concerned with determining the evolutionary relationships between a given set of \textit{taxa}, e.g.\ species.
These relationships are often encoded in a tree in which the leaves represent present-day species and the internal vertices are their common ancestors.
The length of an edge in the tree is usually a measure of the amount of mutation that has occurred between the species on either end.
A standard reference for mathematical phylogenetics is \cite{semple2003phylogenetics}.

Weighted tree data appears in the phylogenetic inference literature, for example when applying bootstrapping to biological data.
Bootstrapping is a widely used resampling method that allows one to estimate various statistical measures, e.g. confidence, on the sample without information about the population \cite{efronBootstrap,soltisBootstrap}.
A typical bootstrapping algorithm in phylogenetics takes as input a table indexed by $n$ taxa, e.g.\ species, and $k$ characters, e.g.\ DNA sequences, and resamples with replacement along the character axis to produce $r$ new $n \times k$ taxa-character tables \cite{soltisBootstrap,felsensteinBootstrap}.
Using a fixed phylogenetic reconstruction method, a tree is then generated from each sample and a consensus tree is computed to summarize the~results.

In \cite{makarenkov2010weighted}, the authors argue that certain phylogenetic bootstrap weightings make non-parametric bootstrapping more robust to noise in the data.
However, that paper is concerned with studying bootstrap support of branches, which is the proportion of bootstrap trees in which the branch occurs.
As an interesting future direction, they suggest studying how the weighting affects robustness of the consensus of the bootstrap trees.
In this section, we use weighted (asymmetric tropical) Fermat-Weber points to define a consensus method.

\subsection{The Space of Equidistant Trees}

We begin with a brief overview of the relevant background for equidistant trees and ultrametrics.

A \textit{phylogenetic tree on $n$ leaves} is a rooted $\R$-weighted tree with leaves labeled by $[n]$.
We further assume that our phylogenetic trees contain no degree two vertices other than the root, which we denote by $\rho$.
We will refer to the weights of the edges as \textit{lengths} in order to avoid confusing them with the weights that appear on the data in the weighted Fermat-Weber problem.
We require the internal branch lengths to be positive, but the lengths of branches adjacent to leaves may be non-positive.
A phylogenetic tree is \textit{equidistant} if the path from in the tree from the root to any leaf has the same length.

A phylogenetic tree $T$ induces a \textit{tree distance} on $[n]$, which is the map $d_T: [n] \times [n] \to \R$ where $d_T(i,j)$ is the length of the path from leaf $i$ to leaf $j$ in $T$.
We will identify $d_T$ with the vector in $(d_T(i,j))_{ij} \in \R^N$, where $N = \binom{n}{2}$.
With the assumption that $T$ contains no degree two vertices, it is possible to recover $T$ uniquely from $d_T \in \R^{N}$ \cite[Theorem 7.2.8]{semple2003phylogenetics}.
Thus, we call the set
\[ \mathcal{T}^\prime_n := \{ d \in \R^{N} \mid d = d_T \text{ for some phylogenetic tree } T \text{ on } n \text{ leaves} \} \]
the \textit{space of equidistant phylogenetic trees}. Moreover, $\mathcal{T}^\prime_n$ is a polyhedral fan homeomorphic to the Billera-Holmes-Vogtman (BHV) space of trees \cite{bhv, speyer2004tropical}.
As usual, we will identify $\mathcal{T}^\prime_n$ with its image in $\R^N/\R\1_N$.
Tropical scaling by $\lambda\1_N$ has the effect of increasing the lengths of all branches adjacent to leaves by $\lambda/2$.

A \textit{dissimilarity map} is a function $\delta: [n] \times [n] \to \R$ such that $\delta(i,i) = 0$ and $\delta(i,j) = \delta(j,i)$ for all $i,j \in [n]$.
A dissimilarity map is an \textit{ultrametric} if it satisfies the three-point condition:
\[ \text{ max of } d(i, j), d(i, k), d(j, k) \text{ is achieved at least twice, for all } 1 \leq i < j < k \leq n. \]
Again, we identify $\delta: [n] \times [n] \to \R$ with the vector $(\delta(i,j))_{ij} \in \R^N$.
The connection between dissimilarity maps and equidistant tree distances goes as follows: a dissimilarity map $\delta: [n] \times [n] \to \R$ is an ultrametric if and only if it is an equidistant tree distance \cite[Theorem 7.2.5]{semple2003phylogenetics}.
Thus, the three-point condition gives explicit tropical equations cutting out $\mathcal{T}^\prime_n$.
In \cite{ardila2006bergman}, Ardila and Klivans further prove that $\mathcal{T}^\prime_n$ is equal to (as sets) the Bergman fan of the complete graph, which implies it is a max-tropical linear space.

Any max-tropical linear space is max-tropically convex \cite[Proposition 5.2.8]{ITG}; the analogous  result holds in the min-tropical convention.
Since in the previous section we used min-tropical convex hulls, it is natural for us to make the substitution $\mathcal{T}_n := -\mathcal{T}^\prime_n$, which is min-tropically convex.
In particular, \Cref{thm:main} implies that for $V \subseteq \mathcal{T}_n$, each point in $\FW(V, \bfw)$ represents an equidistant tree on $n$ leaves.

\subsection{Weighted Asymmetric Tropical Consensus}

One fundamental problem in phylogenetic reconstruction is that different reconstruction methods may produce different trees for the input same data.
A \textit{consensus tree} is a tree that summarizes the different reconstructed trees.
The goal of computing a consensus tree is not to compute the true evolutionary tree, but to determine which evolutionary relationships or features present in the reconstructed trees are likely to be present in the true tree.

Following \cite{joswig-com-tropical-medians}, we define a consensus method via the weighted tropical Fermat-Weber points.
The following result is the weighted analogue of \cite[Theorem 18]{joswig-com-tropical-medians}.

\begin{theorem}[{cf.\ \cite[Theorem 18]{joswig-com-tropical-medians}}]
    \label{thm:FW-of-trees}
    Given trees $\{ t_1, \ldots, t_m \} = T \subset \mathcal{T}_n$ and positive real weights $w_1, \ldots, w_m$, the set $\FW(T, \bfw)$ is contained in $\mathcal{T}_n$. Moreover, all trees in the relative interior of $\FW(T, \bfw)$ have the same topology, meaning that they are isomorphic as partially-labeled unweighted graphs.
\end{theorem}

\begin{proof}
    \Cref{thm:containment} and the tropical convexity of $\mathcal{T}_n$ imply the first claim. The second claim follows from \cite[Theorem 3.2]{page2020tropical}, which says that for any cell $C$ of $\CD(T)$, the trees in the relative interior of $C$ share the same topology.
\end{proof}

Thus, we can define a tropical consensus method using the weighted Fermat-Weber points.

\begin{defn}[{cf.\ \cite[Definition 20]{joswig-com-tropical-medians}}]
    \label{defn:consensus}
    Given trees $\{ t_1, \ldots, t_m \} = T \subset \mathcal{T}_n$ and positive real weights $w_1, \ldots, w_m$, the \textit{weighted asymmetric tropical consensus tree} is the (classical) average of the (classical) vertices of~$\FW(T, \bfw)$.
\end{defn}

It is immediate from \Cref{thm:main} that the consensus method defined above is a \textit{tropically convex consensus method}, which simply requires that the consensus tree lies in the tropical convex hull of the data (see \cite[Section 5.2]{joswig-com-tropical-medians}).
Just like the unweighted tropical consensus method, the weighted tropical consensus method is \textit{regular}, in the following sense:
\begin{enumerate}
    \item The consensus of any number of copies of $t$, with any positive real weights, is $t$;
    \item the consensus does not depend on the ordering of the inputs $(t_i, w_i)$;
    \item permuting the taxa of the input tree results in the same permutation of the taxa in the consensus tree.
\end{enumerate}
The above criteria for weighted data are adapted from the definition of regular in \cite{bryant2017can}.
The proof that the weighted tropical consensus method is regular is the same as for the unweighted case, which is given in \cite[Section 5.2]{joswig-com-tropical-medians}.
It is known that regular consensus methods are not ``extension stable" \cite[Theorem 3]{bryant2017can}.

Whether a consensus method behaves well for a certain feature can be measured by whether it is Pareto or co-Pareto with respect to that feature \cite{bryant}.
A consensus method is \textit{Pareto} with respect to a feature $F$ if whenever $F$ is present in every sample, $F$ is also present in the consensus tree.
A consensus method is \textit{co-Pareto} with respect to $F$ if whenever $F$ is present in the consensus tree, $F$ is present in at least one of the samples. 

One feature of interest is the set of \textit{rooted triples} in a tree \cite[Section 1.3]{bryant}, the definition of which we recall now.
Let $T$ be an equidistant tree on $n$ leaves, and let $i,j,k \in [n]$ be three leaves of $T$.
We say that $T$ contains the rooted triple $ij|k$ if $d_T(i,j) < d_T(i,k) = d_T(j,k)$.
That is, the lowest common ancestor of $i$ and $j$ in $T$ is a proper descendant of the lowest common ancestor of $i$ and $k$.
In \cite{joswig-com-tropical-medians}, Comǎneci and Joswig prove that tropically convex consensus methods behave well with respect to \textit{rooted triples}.

\begin{prop}[{Comǎneci and Joswig \cite[Theorem 22]{joswig-com-tropical-medians}}]
    \label{prop:rooted-triples}
    Any tropically convex consensus method is Pareto and co-Pareto on rooted triples.
\end{prop}

In particular, we have the following result for the weighted tropical consensus method.

\begin{cor}
    \label{cor:pareto}
    The weighted asymmetric tropical consensus tree is Pareto and co-Pareto on rooted triples.
\end{cor}

\section{Outlook}
\label{sec:outlook}

In this paper we proved that weighted (asymmetric tropical) Fermat-Weber points satisfy the same tropical convexity properties as their unweighted counterparts, thus generalizing a result of Comǎneci and Joswig \cite[Theorem 4]{joswig-com-tropical-medians}.
Moreover, we defined a weighted analogue of the tropical consensus method defined in \cite{joswig-com-tropical-medians} and showed that it extends the phylogenetics properties of the unweighted version. 
One possible direction for future research is to determine whether weighting bootstrapped trees improves the robustness of asymmetric tropical consensus methods.

In another direction, Comǎneci and Joswig connected the unweighted asymmetric tropical Fermat-Weber problem to a well-known optimization problem called the transportation problem \cite{joswig-com-tropical-medians}.
This allows them to provide a computationally efficient algorithm for computing the Fermat-Weber set \cite[Corollary 17]{joswig-com-tropical-medians}.
In phylogenetics, efficiency is crucial due to the size of the data sets involved, and because the ambient dimension of tree space grows quadratically in the number of taxa.
Thus, another possible direction for future research is to establish a connection between weighted asymmetric tropical Fermat-Weber points and transportation problems.

\subsection*{Acknowledgments}

We are grateful to Elizabeth Gross, Michael Joswig, David Speyer, and Ruriko Yoshida for helpful conversations, and to the anonymous referees and associate editor whose feedback greatly improved the paper. Many thanks also to Ikenna Nometa, John Sabol, and Jane Coons who read earlier versions of the manuscript.

This work was started at the ``Algebra of phylogenetic networks" workshop held at the University of Hawai`i at Mānoa from May 23 - 27, 2022 which was supported by the National Science Foundation under grant DMS-1945584. The first author was supported by National Science Foundation Graduate Research Fellowship under Grant No. DGE-1841052, and by the National Science Foundation under Grant No. 1855135.

\bibliographystyle{plain}
\bibliography{ref}

\begin{thebibliography}{10}

\bibitem{allamigeon2018log}
Xavier Allamigeon, Pascal Benchimol, St{\'e}phane Gaubert, and Michael Joswig.
\newblock Log-barrier interior point methods are not strongly polynomial.
\newblock {\em SIAM Journal on Applied Algebra and Geometry}, 2(1):140--178, 2018.

\bibitem{ardila2006bergman}
Federico Ardila and Caroline~J Klivans.
\newblock The bergman complex of a matroid and phylogenetic trees.
\newblock {\em Journal of Combinatorial Theory, Series B}, 96(1):38--49, 2006.

\bibitem{bhv}
Louis~J Billera, Susan~P Holmes, and Karen Vogtmann.
\newblock Geometry of the space of phylogenetic trees.
\newblock {\em Advances in Applied Mathematics}, 27(4):733--767, 2001.

\bibitem{bryant}
David Bryant.
\newblock A classifcation of consensus methods for phylogenetics.
\newblock {\em BioConsensus, DIMACS}, 61, 05 2002.

\bibitem{bryant2017can}
David Bryant, Andrew Francis, and Mike Steel.
\newblock Can we “future-proof” consensus trees?
\newblock {\em Systematic biology}, 66(4):611--619, 2017.

\bibitem{joswig-com-tropical-medians}
Andrei Com{\u{a}}neci and Michael Joswig.
\newblock Tropical medians by transportation.
\newblock {\em Mathematical Programming}, pages 1--27, 2023.

\bibitem{de2014phylogenetic}
Alexandre De~Bruyn, Darren~P Martin, and Pierre Lefeuvre.
\newblock Phylogenetic reconstruction methods: an overview.
\newblock {\em Molecular Plant Taxonomy: Methods and Protocols}, pages 257--277, 2014.

\bibitem{triangulations-textbook}
Jes\'{u}s~A. De~Loera, J\"{o}rg Rambau, and Francisco Santos.
\newblock {\em Triangulations}, volume~25 of {\em Algorithms and Computation in Mathematics}.
\newblock Springer-Verlag, Berlin, 2010.
\newblock Structures for algorithms and applications.

\bibitem{develin2004tropical}
Mike Develin and Bernd Sturmfels.
\newblock Tropical convexity.
\newblock {\em Documenta Mathematica}, 9:1--27, 2004.

\bibitem{efronBootstrap}
B.~Efron.
\newblock Bootstrap methods: another look at the jackknife.
\newblock {\em Ann. Statist.}, 7(1):1--26, 1979.

\bibitem{felsensteinBootstrap}
Joseph Felsenstein.
\newblock Confidence limits on phylogenies: An approach using the bootstrap.
\newblock {\em Evolution}, 39(4):783--791, 1985.

\bibitem{hrs-cayley-trick}
Birkett Huber, J\"{o}rg Rambau, and Francisco Santos.
\newblock The {C}ayley trick, lifting subdivisions and the {B}ohne-{D}ress theorem on zonotopal tilings.
\newblock {\em J. Eur. Math. Soc. (JEMS)}, 2(2):179--198, 2000.

\bibitem{ETC}
Michael Joswig.
\newblock {\em Essentials of tropical combinatorics}, volume 219 of {\em Graduate Studies in Mathematics}.
\newblock American Mathematical Society, Providence, RI, [2021] \copyright 2021.

\bibitem{lin2017convexity}
Bo~Lin, Bernd Sturmfels, Xiaoxian Tang, and Ruriko Yoshida.
\newblock Convexity in tree spaces.
\newblock {\em SIAM Journal on Discrete Mathematics}, 31(3):2015--2038, 2017.

\bibitem{lin-yoshida-tropical-fw}
Bo~Lin and Ruriko Yoshida.
\newblock Tropical {F}ermat-{W}eber points.
\newblock {\em SIAM J. Discrete Math.}, 32(2):1229--1245, 2018.

\bibitem{ITG}
Diane Maclagan and Bernd Sturmfels.
\newblock {\em Introduction to tropical geometry}, volume 161 of {\em Graduate Studies in Mathematics}.
\newblock American Mathematical Society, Providence, RI, 2015.

\bibitem{makarenkov2010weighted}
Vladimir Makarenkov, Alix Boc, Jingxin Xie, Pedro Peres-Neto, Fran{\c{c}}ois-Joseph Lapointe, and Pierre Legendre.
\newblock Weighted bootstrapping: a correction method for assessing the robustness of phylogenetic trees.
\newblock {\em BMC evolutionary biology}, 10:1--16, 2010.

\bibitem{page2020tropical}
Robert Page, Ruriko Yoshida, and Leon Zhang.
\newblock Tropical principal component analysis on the space of phylogenetic trees.
\newblock {\em Bioinformatics}, 36(17):4590--4598, 2020.

\bibitem{schrijver2003combinatorial}
Alexander Schrijver et~al.
\newblock {\em Combinatorial optimization: polyhedra and efficiency}, volume~24.
\newblock Springer, 2003.

\bibitem{semple2003phylogenetics}
Charles Semple and Mike Steel.
\newblock {\em Phylogenetics}, volume~24 of {\em Oxford Lecture Series in Mathematics and its Applications}.
\newblock Oxford University Press, Oxford, 2003.

\bibitem{soltisBootstrap}
Pamela~S. Soltis and Douglas~E. Soltis.
\newblock Applying the bootstrap in phylogeny reconstruction.
\newblock {\em Statist. Sci.}, 18(2):256--267, 2003.
\newblock Silver anniversary of the bootstrap.

\bibitem{speyer2004tropical}
David Speyer and Bernd Sturmfels.
\newblock The tropical {G}rassmannian.
\newblock {\em Adv. Geom.}, 4(3):389--411, 2004.

\bibitem{sturmfels-cayley-trick}
Bernd Sturmfels.
\newblock On the {N}ewton polytope of the resultant.
\newblock {\em J. Algebraic Combin.}, 3(2):207--236, 1994.

\end{thebibliography}

\end{document}